\documentclass{article}
\usepackage[utf8]{inputenc}
\usepackage{amsmath}
\usepackage{amsthm}
\usepackage{amssymb}
\usepackage[matrix,arrow,curve]{xy}
\usepackage{tikz}

\newcommand{\p}{\mathbb{P}}

\newtheorem{theorem}{Theorem} 
\newtheorem{lemma}{Lemma} 
\newtheorem{proposition}{Proposition} 
\newtheorem{corollary}{Corollary} 
\begin{document}

\title{Endomorphisms of projective bundles over a certain class of varieties}

\author{Ekaterina Amerik, Alexandra Kuznetsova}
\date{\today}
\maketitle

\section*{Introduction}

During the last 20 years, the question which smooth 
projective varieties have endomorphisms of degree greater than one (which we shall
sometimes simply call ``endomorphisms'', as opposed to automorphisms) has 
attracted some attention for both geometric and dynamical reasons (see e.g.
\cite{ARV}, \cite{B}, \cite{N}, \cite{NZ} - this is only a beginning of the list).
Though in this generality it is still far from being solved, there is a number
of partial results suggesting that varieties with such endomorphisms generally 
come from two obvious cases (tori and toric varieties) by means of simple 
geometric constructions such as taking a product with another smooth projective variety or taking a quotient by a finite freely acting group. For instance,
Nakayama proved in the beginning of 2000's that a rational smooth projective surface with endomorphisms 
must be toric. Around the same time, one of the authors of the present note 
has considered the case 
of a projective bundle $X$ over a projective base $B$, $p:X\to B$, and proved 
that if $X$ has
an endomorphism commuting with the projection onto the base, then $X$ must be
a quotient of a product $B'\times \p^r$ by a finite freely acting group.
A simple remark on endomorphisms of projective bundles $X=\p(E)$, where $E$ is a
vector bundle (\cite{A}, p. 18) is that a power  
of any 
$f:X\to X$ sends fibers to fibers and thus must be over an endomorphism of 
the base; so if by any chance we know that all endomorphisms of $B$ are of 
finite order - for instance when $B$ is of general type - then this result 
describes the situation completely.

The argument (the ``only if'' part, the ``if'' part being rather standard) proceeds as follows. One considers the space of all morphisms $R^m(\p(V))$
from $\p^n=\p(V)$ to itself given by degree $m$ polynomials (well-known to be an affine variety) 
and its quotient $M$ by $PGL(V)$ (that is, the spectrum of the ring of the invariants).
It turns out that for $m$ big enough, $PGL(V)$ acts with finite stabilizers, so $M$ is the geometric quotient
(i.e. actually parameterizes the orbits of the action). Now let $X=\p(E)$ be a projective bundle
over $B$. An endomorphism $f$ of $X$ over $B$ naturally induces a morphism from $B$ to $M$. Its image
must be a point since $B$ is projective and $M$ is affine. Let $t$ be a lift of this point to $R^m(\p(V))$.
Over a suitably fine open covering $(U_{\alpha})$ of $B$ we have $f_{\alpha}=h_{\alpha}\cdot t$, where $h_{\alpha}$ is in $PGL_{n+1}({\mathcal O}_{U_{\alpha}})$. Denote by 
$g_{\alpha \beta}$ the transition functions of our projective bundle,
it follows that $h_{\alpha}^{-1}g_{\alpha \beta}h_{\beta}\in Stab(t)$, in other words, by changing the trivialization we make the transition functons of 
$X$ constant with values in a finite group, q. e. d..

In general, for an endomorphism $f$ of $\p(E)$ we may suppose that $f$ is over an endomorphism $g$ of the base;
there are then two cases to be treated: the case where $f$ induces isomorphisms
of fibers (considered as exceptional; when $X=\p(E)$ it means that $g^*E$ is a
shift of $E$ by a line bundle) and the case where the degree of $f$ on the
fibers is greater than one. In \cite{A}, only the rank-two case 
(that of projective line bundles) was considered.
It was established that either $X$ is a finite quotient of a product or $E$ has
a subbundle. This last statement has been pursued further to yield that $E$ must
split into a direct sum of line bundles after a finite, not necessarily \'etale, base change (\cite{A}, theorem 2);
from a different point of view, one can restrict to a specific class of bases 
to obtain a stronger statement. For instance, if $B$ satisfies the condition
$H^1(B, L)=0$ for any line bundle $L$, then having a subbundle is equivalent 
to splitting for rank-two bundles. It therefore follows from the results of \cite{A}
that if $B$ is simply connected and $H^1(B, L)=0$ for any line bundle $L$ on $B$, then an $X$ with 
endomorphisms of degree greater than one on fibers must be the projectivization of a split rank-two bundle.
 

The purpose of the present note is to prove this result in the case of arbitrary rank 
projective bundles over such specific bases.

 \begin{theorem}\label{Endomorphisms}
  Let $B$ be a simply-connected projective variety such that for any line bundle $\mathcal{L}$ its first cohomology $H^1(B, \mathcal{L}) = 0$. 
Let $E$ be a vector bundle of rank $n+1$ on  $B$. If there exists 
  a fiberwise endomorphism 
  \begin{equation}
  \xymatrix{ \p(E) \ar[r]^{\phi} \ar[d]^{\pi}& \p(E)\ar[d]^{\pi}\\
  B \ar[r]^{\Phi} & B}
  \end{equation}
  of degree greater than one on the fibers, then $E$ splits into a direct sum of line bundles:
  \begin{equation}
   E = \mathop{\oplus}\limits_{i=0}^{n+1} \mathcal{L}_i
  \end{equation}
 \end{theorem} 

What we show is in fact slightly more general, as in \cite{A}.

 \begin{theorem}\label{Morphism}
  Let $B$ be as in the previous theorem and $E$ and $F$ vector bundles of rank $n+1$ on  $B$. If there exists a morphism
  \begin{equation}
   \phi : \p(E) \to \p(F)
  \end{equation}
  over $B$ which is of degree greater than one, then $E$ and $F$ both split into a direct sum of line bundles.  
 \end{theorem} 
 
 Obviously, theorem \ref{Endomorphisms} follows from this statement: consider the endomorphism $\phi$ as a morphism $\p(E) \to \p(\Phi^*(E))$, and apply the theorem \ref{Morphism}.

 In the ideal situation, one would like to prove the statement of Theorem \ref{Endomorphisms} for
an arbitrary toric base $B$. The reason is that the projectivization of a vector bundle over a toric base
is itself toric if and only if the bundle is split (\cite{D}). This would therefore strongly support the
principle that varieties with endomorphisms are closely related to toric varieties or tori.
However few toric bases (e.g. $\p^n, n\geq 2$, or products of such) actually satisfy the cohomology 
vanishing condition as above; so more work is needed to obtain such a result. It is certainly related
to the fact that we never make use of a condition like $F=g^*E$ in Theorem \ref{Morphism}.

 \section{Reduction to invariant theory}

  Let $V$ and $W$ be vector spaces of dimension $n+1$. Denote by $R^m(\p(V),\p(W))$ the set of all morphisms between $\p(V)$ and $\p(W)$ given 
by homogeneous polynomials of degree $m$ without a common zero except at $(0,0,\dots, 0)$:

 \begin{equation}
 \begin{split}
 y_0 = f_0(x_0,&x_1, \dots, x_n) \\
  y_1 = f_1(x_0,&x_1, \dots, x_n) \\
   &\cdots \\
  y_n = f_n(x_0,&x_1, \dots, x_n)   
 \end{split}
 \end{equation}

This is an affine variety, indeed the complement to the hypersurface defined by the 
resultant of the $f_i$ in the projective space $\p(Hom(W^*, S^m V^*))$,  with the action of $PGL(V) \times PGL(W)$ given by
 \begin{equation}
   ((g,h)\cdot f) (x) = h^{-1}(f(g(x))).
  \end{equation}

The quotient $M$ of $R^m(\p(V),\p(W))$ by this action (i.e. the spectrum of the ring of invariants), 
in contrast with the case of the action of $PGL(V)$ when $V=W$
(\cite{A}) is not a geometric quotient: indeed some points have infinite stabilizers, and all the adherent orbits give
the same point on the quotient. Let us denote by $M_0$ the ``bad subset'' of $M$ (by definition it consists of points
corresponding to orbits not separated by the invariants). 

When some fiber of a vector bundle $E$ over $B$ is identified with $V$ and that of $F$ with $W$, 
a morphism of projective bundles $\p(E)\to \p(F)$ over a base $B$ gives, in the same way as in \cite{A}, a map from $B$ to $M$, which must be constant
as soon as $B$ is projective. If the image point is not in $M_0$ we conclude as before that $\p(E)$ and $\p(F)$ trivialize
after a finite unramified base change. If $B$ is simply-connected, this yields that these are already trivial on $B$,
and in particular they split into a direct sum of line bundles. So the interesting case is when the image point lands in
$M_0$. In this situation, we strive to deduce some informaion about the geometry of our morphism. We aim to show
that $E$ and $F$ have subbundles $E'$ and $F'$ such that the inverse image of $\p(F')$ is $\p(E')$ and that the map $f$
induces a morphism on the quotients. This shall enable us to conclude by induction in the case when the cohomological
condition on $B$ is satisfied.

Let us also remark that replacing our original endomorphism $\phi$ of $\p(E)$ by a power, we may assume that $m$ is greater than the rank $n+1$ of the vector bundles $E$ and $F$, as we shall for the computations in the next section.

 \section{Unstable morphisms}
  In this section we consider two vector spaces $V$ and $W$ of dimension $n+1$ and a morphism $f$ between their projectivisations of degree $d=m^n$. First of all assume $f$ is stabilized by an infinite subgroup $Stab(f)$ 
  in $PGL(V) \times PGL(W)$. Recall from \cite{A}:
 \begin{lemma}[\cite{A}, Lemma 1.2]
  If $m>n+1$, then a unipotent element $u$ of $PGL(V) \times PGL(W)$ does not stabilize any element of $R^m(\p(V),\p(W))$.
 \end{lemma}
  By this lemma the subgroup $Stab(f) \subset PGL(V) \times PGL(W)$ consists of semisimple elements. Take any of these elements and consider the minimal subgroup in the stabilizer that contains this element. The connected 
 component of the unity of this subgroup is an algebraic torus or trivial. If it is trivial for any element in $Stab(f)$, then the stabilizer is discrete and therefore is finite. If $Stab(f)$ is infinite, it contains a subgroup
 isomorphic to $\mathbb{G}_m$. Lifting its action on $\p(V)$ and $\p(W)$ to an action on $V$ and $W$ we assume that it is given by 
 \begin{equation} \label{diag}
 \begin{split}
  &g_{b,c}: \mathbb{G}_m \to GL(V)\times GL(W)\\
  &g_{b,c}(\lambda) = (diag(\lambda^{c_0},\lambda^{c_1},\dots, \lambda^{c_n});diag(\lambda^{b_0},\lambda^{b_1},\dots, \lambda^{b_n}))
 \end{split}
 \end{equation}
 in appropriate coordinates on $V$ and $W$.
 
 In these coordinates, let the morphism $f$ be given by $(f_0,f_1,\dots,f_n)$ with
 \begin{equation}
 \begin{split}
   y_0 =f_0(x_0,\dots, x_n)=& \sum_{|I|= m} a_{0,I} x^I \\
   y_1 =f_1(x_0,\dots, x_n)=& \sum_{|I|= m} a_{1,I} x^I\\
   &\cdots \\
   y_n =f_n(x_0,\dots, x_n)=& \sum_{|I|= m} a_{n,I} x^I
 \end{split}
 \end{equation}
 Here $I = (i_0,i_1,\dots,i_n)$ is a multiindex and $|I| = i_0 +i_1+\dots+i_n$.
 
 Applying an element of the diagonal group in $g_{b,c}(\lambda)\in Stab(f)$, 
we get the following formulae for $g_{b,c} \cdot f$:
 \begin{equation}\label{action}
  y_j = \sum\lambda^{\langle c ,I\rangle-b_j}a_{j,I} x^I 
 \end{equation}
 Here $\langle -,- \rangle$ denotes the scalar product between multiindexes:
 \begin{equation} \label{level}
  F(I) := \langle c ,I\rangle = \sum_{j=0}^n c_j i_j
 \end{equation}

 Since $g_{b,c}$ stabilizes $f$ there exists a constant $C$, such that for any $j, I$ with $a_{j,I} \ne 0$  
 \begin{equation}\label{log_plane}
  \langle c,I\rangle -b_j = C
 \end{equation}
 
 Consider the $n+1$)--dimensional lattice $\Lambda \cong \mathbb{Z}^{n+1} \subset \mathbb{R}^{n+1}= \Lambda\otimes \mathbb{R}$. Denote by $p_i\in \Lambda$ the vertex corresponding to the $i-$th base vector $(0,\dots, 0,m,0,\dots,0)$. For any subset
 $\{p_{i_1}, \dots, p_{i_k}\}$ denote by $\Delta(p_{i_1}, \dots, p_{i_k})\subset \Lambda$ the simplex of dimension $k-1$ with vertexes $p_{i_1}, \dots, p_{i_k}$. Set
 \begin{equation}
  \Delta = \Delta(p_0,p_1,\dots,p_n) \subset \Lambda
 \end{equation}

 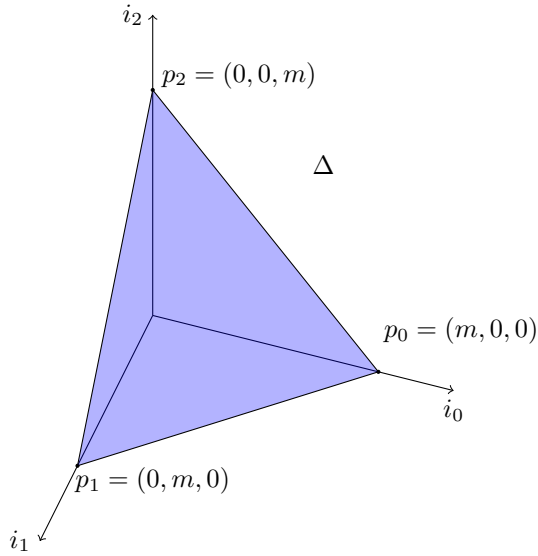
\begin{figure}[h]
\begin{center}
 \begin{tikzpicture}
\draw [->] (0,0) -- (0,4);
\draw [->] (0,0) -- (-1.5,-3);
\draw [->] (0,0) -- (4,-1);

\node[right] at (0,3.2) {$p_2 = (0,0,m)$};
\node[below] at (0,-1.9) {$p_1=(0,m,0)$};
\node[above] at (4.1,-0.5) {$p_0=(m,0,0)$};

\node[right] at (2,2) {$\Delta$};

 \node[left] at (0,4) {$i_2$};
 \node[left] at (-1.5,-3) {$i_1$};
 \node[below] at (4,-1) {$i_0$};

 \draw[fill] (0,3) circle [radius =0.02];
 \draw[fill] (-1,-2) circle [radius =0.02];
 \draw[fill] (3,-0.75) circle [radius =0.02];
 
 \draw[fill=blue,fill opacity=0.3] (0,3) -- (-1,-2)-- (3, -0.75)-- (0,3);
\end{tikzpicture}
 \caption{The simplex $\Delta$ in the case $n=2$}  
\end{center}
 \end{figure}

 Equations \eqref{log_plane} define $n+1$ hyperplanes in $\mathbb{R}^n$ (not necessarily distinct). Let us denote them by $\Pi_j$.
  
  Now let us consider the Newton polyhedron of $f_j$:
 \begin{equation}
  NP(f_j) := Conv\{I \in \Lambda \mid a_{j,I} \ne 0\}
 \end{equation}
 and prove  some easy facts about Newton polyhedra of the morphism $f$.
 \begin{proposition}
  If $f$ has infinite stabilizer then $NP(f_j)\subset \Pi_j \cap \Delta$.
 \end{proposition}
\begin{proof}
 As the degree of $f_j$ equals $m$, the polyhedron $NP(f_j)$ lies in the simplex $\Delta$. By the previous calculation we see that if  $g_{b,c}$ stabilizes $f$, then \eqref{log_plane} holds and consequently the multi-indices of the monomials of $f_j$ lies in the 
 hyperplane $\Pi_j$.
\end{proof}
\begin{lemma}\label{v_in_p}
  If $f$ is a morphism of projective spaces then every vertex of $\Delta$ is contained in one of the hyperplanes $\Pi_j$.
 \end{lemma}
 \begin{proof}
  Assume the vertex $p_0 = (m,0, \dots, 0)$ does not lie in any $\Pi_j$. Consequently no polynomial $f_j$ contains the monomial $x_0^m$. Then all $f_j$ vanish at the point $(1:0: \dots: 0)\in \p(V)$, so $f$ is not a morphism.
 \end{proof}
 \begin{lemma}\label{p_contains_v}
  Each hyperplane $\Pi_j$ contains some vertex of $\Delta$. Moreover a hyperplane repeated exactly $k+1$ times
(i.e. corresponding to the polynomials $f_0, \dots, f_k$, up to renumbering) contains exactly $k+1$ vertices
of $\Delta$.
 \end{lemma}
 \begin{proof}

 Since all the hyperplanes $\Pi_j$ are parallel, if they contain a common vertex they coincide. There is
a natural partition of the set $H(f)$ of equations 

\begin{equation}\label{union_of_Pi}
    H_1 \sqcup H_2 \sqcup \dots \sqcup H_l = H(f)
  \end{equation}
where a subset $H_i$ consists of equations corresponding to the same hyperplane $\Pi_i$,
as well as of the set of vertices 

 \begin{equation}
   V(\Delta) = V_1 \sqcup V_2 \sqcup \dots \sqcup V_l
  \end{equation}
where $V_i$ consists of vertices lying in $\Pi_j$.

 Since $|V(\Delta)|=n+1=|H(f)|$ it follows that either the statement of the lemma is true or 
$k+1=|V_i|>|H_i|=s+1$ for some $i$.

  Assume $|V_i|>|H_i|$. The polynomials $f_i$ indexed by $H_i$ contain monomials depending only on the 
variables indexed by $V_i$,
but the others do not: up to renumbering, $f_{s+1}, \dots, f_n$ are zero as soon as $x_{k+1}=\dots x_n=0$.
Then $f_0,\dots, f_s$ define a regular map of the subspace of $\p(V)$ given by the vanishing of 
$x_{k+1},\dots x_n$ to the subspace of $\p(W)$ given by the vanishing of $y_{s+1}, \dots, y_n$, but this is impossible
since the dimension of the source would then be greater than that of the target.


 \end{proof} 
 From these assertions we deduce the following statement.
 \begin{proposition} \label{NP_inf_stab_orbits}
  Let $f$ be a morphism between $\p(V)$ and $\p(W)$ with infinite stabilizer in $PGL(V) \times PGL(W)$. There exist $V'\subset V(\Delta)$ and $H'\subset H(f)$ such that 
$|V'|=|H'|<n+1$ and 
  \begin{equation}
   NP(f_j)\subset \Delta(V')
   \end{equation}
  for any $f_j \in H'$.
 \end{proposition}
 \begin{proof}
  Let us recall the function $F$ from \eqref{level}. Denote $M' = \max \{F(p_i)\}$, where $p_i$ runs through the set of vertices of $\Delta$. Set 
  \begin{equation}
   H'= \{f_j\mid F|_{\Pi_j}=M'\}
  \end{equation}
  As $F$ is not constant on $\Delta$, $\emptyset\subsetneq H'\subsetneq H(f)$. Denote by $V'$ the set of vertices on the hyperplane corresponding to the equations in $H'$. By the previous lemma $|V'|=|H'|$.
  Obviously, $\Pi_j \cap \Delta = \Delta(V')$ and so the polynomials  $f_j\in H'$ depend
only on the variables corresponding to the vertices in $V'$.  
  \end{proof}
  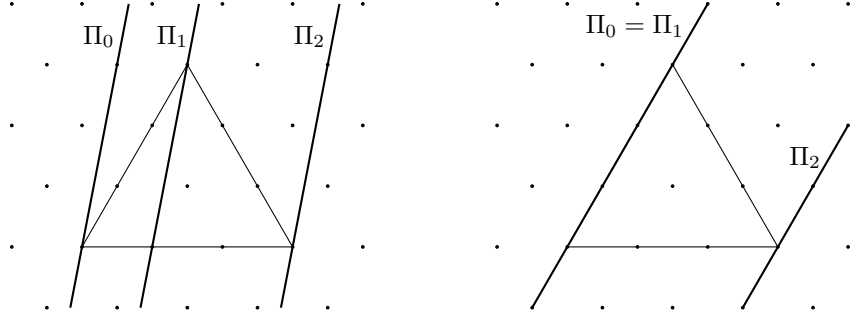
\begin{figure}
 \begin{center}
 \begin{tikzpicture}[scale=0.7]
  \draw (0,0) -- (2,3.464) -- (4,0) -- (0,0);
 
  \draw[thick] (1.111,-1.155) -- (2.222,4.619);
  \draw[thick] (-0.222,-1.155) -- (0.888,4.619);
  \draw[thick] (3.778,-1.155) -- (4.889,4.619);
  
  \node[left] at (0.8, 4) {$\Pi_0$}; 
  \node[left] at (2.2, 4) {$\Pi_1$}; 
  \node[left] at (4.8, 4) {$\Pi_2$};

  \draw[fill] (-1.333,0) circle [radius =0.025];
  \draw[fill] (0,0) circle [radius =0.025];
  \draw[fill] (1.333,0) circle [radius =0.025];
  \draw[fill] (2.667,0) circle [radius =0.025];
  \draw[fill] (4,0) circle [radius =0.025];
  \draw[fill] (5.333,0) circle [radius =0.025];

  \draw[fill] (-0.667,-1.155) circle [radius =0.025];
  \draw[fill] (0.667,-1.155) circle [radius =0.025];
  \draw[fill] (2,-1.155) circle [radius =0.025];
  \draw[fill] (3.333,-1.155) circle [radius =0.025];
  \draw[fill] (4.667,-1.155) circle [radius =0.025];
 
  \draw[fill] (-1.333,2.309) circle [radius =0.025];
  \draw[fill] (0,2.309) circle [radius =0.025];
  \draw[fill] (1.333,2.309) circle [radius =0.025];
  \draw[fill] (2.667,2.309) circle [radius =0.025];
  \draw[fill] (4,2.309) circle [radius =0.025];
  \draw[fill] (5.333,2.309) circle [radius =0.025];

  \draw[fill] (-0.667,1.155) circle [radius =0.025];
  \draw[fill] (0.667,1.155) circle [radius =0.025];
  \draw[fill] (2,1.155) circle [radius =0.025];
  \draw[fill] (3.333,1.155) circle [radius =0.025];
  \draw[fill] (4.667,1.155) circle [radius =0.025];

  \draw[fill] (-1.333,4.619) circle [radius =0.025];
  \draw[fill] (0,4.619) circle [radius =0.025];
  \draw[fill] (1.333,4.619) circle [radius =0.025];
  \draw[fill] (2.667,4.619) circle [radius =0.025];
  \draw[fill] (4,4.619) circle [radius =0.025];
  \draw[fill] (5.333,4.619) circle [radius =0.025];

  \draw[fill] (-0.667,3.464) circle [radius =0.025];
  \draw[fill] (0.667,3.464) circle [radius =0.025];
  \draw[fill] (2,3.464) circle [radius =0.025];
  \draw[fill] (3.333,3.464) circle [radius =0.025];
  \draw[fill] (4.667,3.464) circle [radius =0.025];
  
 \end{tikzpicture}
\hspace{1.5cm}  \begin{tikzpicture}[scale=0.7]
  \draw (0,0) -- (2,3.464) -- (4,0) -- (0,0);
 
  \draw[thick] (-0.667,-1.155) -- (2.667,4.619);
  \draw[thick] (3.333,-1.155) -- (5.333,2.309); 
  
  \node[left] at (2.4,4.2) {$\Pi_0=\Pi_1$};
  \node[left] at (5, 1.7) {$\Pi_2$};
 
  \draw[fill] (-1.333,0) circle [radius =0.025];
  \draw[fill] (0,0) circle [radius =0.025];
  \draw[fill] (1.333,0) circle [radius =0.025];
  \draw[fill] (2.667,0) circle [radius =0.025];
  \draw[fill] (4,0) circle [radius =0.025];
  \draw[fill] (5.333,0) circle [radius =0.025];

  \draw[fill] (-0.667,-1.155) circle [radius =0.025];
  \draw[fill] (0.667,-1.155) circle [radius =0.025];
  \draw[fill] (2,-1.155) circle [radius =0.025];
  \draw[fill] (3.333,-1.155) circle [radius =0.025];
  \draw[fill] (4.667,-1.155) circle [radius =0.025];
 
  \draw[fill] (-1.333,2.309) circle [radius =0.025];
  \draw[fill] (0,2.309) circle [radius =0.025];
  \draw[fill] (1.333,2.309) circle [radius =0.025];
  \draw[fill] (2.667,2.309) circle [radius =0.025];
  \draw[fill] (4,2.309) circle [radius =0.025];
  \draw[fill] (5.333,2.309) circle [radius =0.025];

  \draw[fill] (-0.667,1.155) circle [radius =0.025];
  \draw[fill] (0.667,1.155) circle [radius =0.025];
  \draw[fill] (2,1.155) circle [radius =0.025];
  \draw[fill] (3.333,1.155) circle [radius =0.025];
  \draw[fill] (4.667,1.155) circle [radius =0.025];

  \draw[fill] (-1.333,4.619) circle [radius =0.025];
  \draw[fill] (0,4.619) circle [radius =0.025];
  \draw[fill] (1.333,4.619) circle [radius =0.025];
  \draw[fill] (2.667,4.619) circle [radius =0.025];
  \draw[fill] (4,4.619) circle [radius =0.025];
  \draw[fill] (5.333,4.619) circle [radius =0.025];

  \draw[fill] (-0.667,3.464) circle [radius =0.025];
  \draw[fill] (0.667,3.464) circle [radius =0.025];
  \draw[fill] (2,3.464) circle [radius =0.025];
  \draw[fill] (3.333,3.464) circle [radius =0.025];
  \draw[fill] (4.667,3.464) circle [radius =0.025];
  
 \end{tikzpicture}

\end{center}
  \caption{Two types of Newton polyhedra of $f_0$, $f_1$ and $f_2$ in the case $n=2$.} 
 \end{figure}

 So far, we have discussed the morphisms of projective spaces with infinite stabilizer in $PGL(V)\times PGL(W)$. But our goal is to study the morphisms $f$ with non-closed orbits under the group action. By a generalization 
 of the Hilbert--Mumford criterion (\cite{Hilbert-Mumford} Theorem 4.2), we reach the boundary of the orbit $(PGL(V)\times PGL(W))\cdot f$ while acting on $f$ by one-parameter subgroups $g_{b,c}(\mathbb{G}_m)$ as in 
 \eqref{diag}. As earlier, the map $g_{b,c}(\lambda)\cdot f$ is
given by the equations \eqref{action}.
  Let us introduce a new notation
  \begin{equation}
  K_j = min_{\{I| a_{j,I} \ne 0 \} } \{ \langle c,I\rangle-b_j\} 
 \end{equation}
 Set $K= min_j \{K_j\}$. Then we can describe the limit of $g_{b,c}(\lambda)\cdot f$ when $\lambda$ goes to zero.
 \begin{lemma}\label{unst-eqs}
  Denote $\bar{f} =  \lim_{\lambda \to 0}(g_{b,c}(\lambda)\cdot f)$, then
 \begin{equation}
  \bar{f}_j(x_0,\dots,x_n)=\sum_{\langle c,I\rangle-b_j = K} a_{j,I} x^I 
 \end{equation}
 and the original map was of type:
 \begin{equation}
  f_j = \sum_{\langle c,I\rangle-b_0 = K} a_{j,I} x^I + \sum_{\langle c,I\rangle-b_0 > K} a_{j,I} x^I
 \end{equation}
 \end{lemma}
 
The proof is a straightforward calculation.

 Obviously, the group $G_{b,c}$ stabilize the morphism $\bar{f}$, so $\bar{f}$ has infinite stabilizer and in Proposition \ref{NP_inf_stab_orbits} we have a description of its Newton polyhedron. Now consider the set of half-spaces
 \begin{equation}
  \Pi_j^+ = \{I\in \Lambda\otimes \mathbb{R}\mid \langle c,I\rangle-b_j\ge K \subset \mathbb{R}^n\}
 \end{equation}
 Lemma \ref{unst-eqs} implies that $NP(f_j)= \Pi_j^+\cap \Delta$. From the proof of Proposition \ref{NP_inf_stab_orbits} we see that there is always a hyperplane $\Pi_j$
intersecting our simplex $\Delta$ by a face and such that the rest of the simplex is 
below $\Pi_j$. Thus the following holds.

 \begin{proposition} \label{NP_non_closed_orbits}
  If $f$ is an unstable morphism between $\p(V)$ and $\p(W)$, then there are nonempty sets  $V'\subset V(\Delta)$ and $H'\subset H(f)$ such that $|V'|=|H'|<n+1$ and 
  \begin{equation}
   NP(f_j)\subset \Delta(V')
   \end{equation}
  for any $f_j \in H'$.
 \end{proposition}
 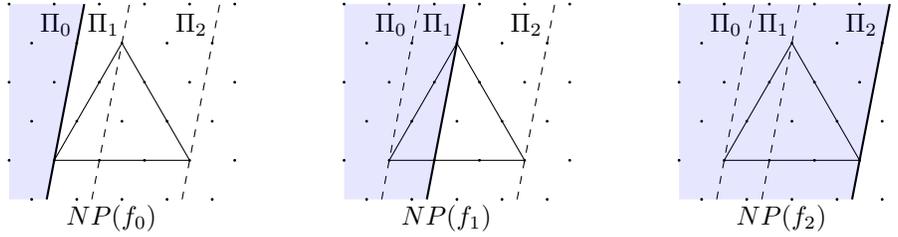
\begin{figure}[h]
 \begin{center}
 \begin{tikzpicture}[scale=0.45]
  \node[left] at (3.3	,-1.7) {$NP(f_0)$};
  \draw (0,0) -- (2,3.464) -- (4,0) -- (0,0);
 
  \draw[dashed] (1.111,-1.155) -- (2.222,4.619);
  \draw[thick] (-0.222,-1.155) -- (0.888,4.619);
  \draw[dashed] (3.778,-1.155) -- (4.889,4.619);
  
   \node[left] at (0.8, 4) {$\Pi_0$}; 
  \node[left] at (2.2, 4) {$\Pi_1$}; 
  \node[left] at (4.8, 4) {$\Pi_2$}; 
 
  \path[fill=blue, fill opacity=0.1] (-1.333,-1.155) -- (-0.222,-1.155) -- (0.888,4.619) -- (-1.333, 4.619);

  \draw[fill] (-1.333,0) circle [radius =0.025];
  \draw[fill] (0,0) circle [radius =0.025];
  \draw[fill] (1.333,0) circle [radius =0.025];
  \draw[fill] (2.667,0) circle [radius =0.025];
  \draw[fill] (4,0) circle [radius =0.025];
  \draw[fill] (5.333,0) circle [radius =0.025];

  \draw[fill] (-0.667,-1.155) circle [radius =0.025];
  \draw[fill] (0.667,-1.155) circle [radius =0.025];
  \draw[fill] (2,-1.155) circle [radius =0.025];
  \draw[fill] (3.333,-1.155) circle [radius =0.025];
  \draw[fill] (4.667,-1.155) circle [radius =0.025];
 
  \draw[fill] (-1.333,2.309) circle [radius =0.025];
  \draw[fill] (0,2.309) circle [radius =0.025];
  \draw[fill] (1.333,2.309) circle [radius =0.025];
  \draw[fill] (2.667,2.309) circle [radius =0.025];
  \draw[fill] (4,2.309) circle [radius =0.025];
  \draw[fill] (5.333,2.309) circle [radius =0.025];

  \draw[fill] (-0.667,1.155) circle [radius =0.025];
  \draw[fill] (0.667,1.155) circle [radius =0.025];
  \draw[fill] (2,1.155) circle [radius =0.025];
  \draw[fill] (3.333,1.155) circle [radius =0.025];
  \draw[fill] (4.667,1.155) circle [radius =0.025];

  \draw[fill] (-1.333,4.619) circle [radius =0.025];
  \draw[fill] (0,4.619) circle [radius =0.025];
  \draw[fill] (1.333,4.619) circle [radius =0.025];
  \draw[fill] (2.667,4.619) circle [radius =0.025];
  \draw[fill] (4,4.619) circle [radius =0.025];
  \draw[fill] (5.333,4.619) circle [radius =0.025];

  \draw[fill] (-0.667,3.464) circle [radius =0.025];
  \draw[fill] (0.667,3.464) circle [radius =0.025];
  \draw[fill] (2,3.464) circle [radius =0.025];
  \draw[fill] (3.333,3.464) circle [radius =0.025];
  \draw[fill] (4.667,3.464) circle [radius =0.025];
  
 \end{tikzpicture}
\hspace{1.3cm}\begin{tikzpicture}[scale=0.45]
  \node[left] at (3.3	,-1.7) {$NP(f_1)$};
  \draw (0,0) -- (2,3.464) -- (4,0) -- (0,0);
 
  \draw[thick] (1.111,-1.155) -- (2.222,4.619);
  \draw[dashed] (-0.222,-1.155) -- (0.888,4.619);
  \draw[dashed] (3.778,-1.155) -- (4.889,4.619);
  
   \node[left] at (0.8, 4) {$\Pi_0$}; 
  \node[left] at (2.2, 4) {$\Pi_1$}; 
  \node[left] at (4.8, 4) {$\Pi_2$}; 
 
  \path[fill=blue, fill opacity=0.1] (-1.333,-1.155) -- (1.111,-1.155) -- (2.222,4.619) -- (-1.333, 4.619);

  \draw[fill] (-1.333,0) circle [radius =0.025];
  \draw[fill] (0,0) circle [radius =0.025];
  \draw[fill] (1.333,0) circle [radius =0.025];
  \draw[fill] (2.667,0) circle [radius =0.025];
  \draw[fill] (4,0) circle [radius =0.025];
  \draw[fill] (5.333,0) circle [radius =0.025];

  \draw[fill] (-0.667,-1.155) circle [radius =0.025];
  \draw[fill] (0.667,-1.155) circle [radius =0.025];
  \draw[fill] (2,-1.155) circle [radius =0.025];
  \draw[fill] (3.333,-1.155) circle [radius =0.025];
  \draw[fill] (4.667,-1.155) circle [radius =0.025];
 
  \draw[fill] (-1.333,2.309) circle [radius =0.025];
  \draw[fill] (0,2.309) circle [radius =0.025];
  \draw[fill] (1.333,2.309) circle [radius =0.025];
  \draw[fill] (2.667,2.309) circle [radius =0.025];
  \draw[fill] (4,2.309) circle [radius =0.025];
  \draw[fill] (5.333,2.309) circle [radius =0.025];

  \draw[fill] (-0.667,1.155) circle [radius =0.025];
  \draw[fill] (0.667,1.155) circle [radius =0.025];
  \draw[fill] (2,1.155) circle [radius =0.025];
  \draw[fill] (3.333,1.155) circle [radius =0.025];
  \draw[fill] (4.667,1.155) circle [radius =0.025];

  \draw[fill] (-1.333,4.619) circle [radius =0.025];
  \draw[fill] (0,4.619) circle [radius =0.025];
  \draw[fill] (1.333,4.619) circle [radius =0.025];
  \draw[fill] (2.667,4.619) circle [radius =0.025];
  \draw[fill] (4,4.619) circle [radius =0.025];
  \draw[fill] (5.333,4.619) circle [radius =0.025];

  \draw[fill] (-0.667,3.464) circle [radius =0.025];
  \draw[fill] (0.667,3.464) circle [radius =0.025];
  \draw[fill] (2,3.464) circle [radius =0.025];
  \draw[fill] (3.333,3.464) circle [radius =0.025];
  \draw[fill] (4.667,3.464) circle [radius =0.025];
  
 \end{tikzpicture}
\hspace{1.3cm}\begin{tikzpicture}[scale=0.45]
  \node[left] at (3.3	,-1.7) {$NP(f_2)$};
  \draw (0,0) -- (2,3.464) -- (4,0) -- (0,0);
 
  \draw[dashed] (1.111,-1.155) -- (2.222,4.619);
  \draw[dashed] (-0.222,-1.155) -- (0.888,4.619);
  \draw[thick] (3.778,-1.155) -- (4.889,4.619);
 
  \node[left] at (0.8, 4) {$\Pi_0$}; 
  \node[left] at (2.2, 4) {$\Pi_1$}; 
  \node[left] at (4.8, 4) {$\Pi_2$}; 
 
  \path[fill=blue, fill opacity=0.1] (-1.333,-1.155) -- (3.778,-1.155) -- (4.889,4.619) -- (-1.333, 4.619);

  \draw[fill] (-1.333,0) circle [radius =0.025];
  \draw[fill] (0,0) circle [radius =0.025];
  \draw[fill] (1.333,0) circle [radius =0.025];
  \draw[fill] (2.667,0) circle [radius =0.025];
  \draw[fill] (4,0) circle [radius =0.025];
  \draw[fill] (5.333,0) circle [radius =0.025];

  \draw[fill] (-0.667,-1.155) circle [radius =0.025];
  \draw[fill] (0.667,-1.155) circle [radius =0.025];
  \draw[fill] (2,-1.155) circle [radius =0.025];
  \draw[fill] (3.333,-1.155) circle [radius =0.025];
  \draw[fill] (4.667,-1.155) circle [radius =0.025];
 
  \draw[fill] (-1.333,2.309) circle [radius =0.025];
  \draw[fill] (0,2.309) circle [radius =0.025];
  \draw[fill] (1.333,2.309) circle [radius =0.025];
  \draw[fill] (2.667,2.309) circle [radius =0.025];
  \draw[fill] (4,2.309) circle [radius =0.025];
  \draw[fill] (5.333,2.309) circle [radius =0.025];

  \draw[fill] (-0.667,1.155) circle [radius =0.025];
  \draw[fill] (0.667,1.155) circle [radius =0.025];
  \draw[fill] (2,1.155) circle [radius =0.025];
  \draw[fill] (3.333,1.155) circle [radius =0.025];
  \draw[fill] (4.667,1.155) circle [radius =0.025];

  \draw[fill] (-1.333,4.619) circle [radius =0.025];
  \draw[fill] (0,4.619) circle [radius =0.025];
  \draw[fill] (1.333,4.619) circle [radius =0.025];
  \draw[fill] (2.667,4.619) circle [radius =0.025];
  \draw[fill] (4,4.619) circle [radius =0.025];
  \draw[fill] (5.333,4.619) circle [radius =0.025];

  \draw[fill] (-0.667,3.464) circle [radius =0.025];
  \draw[fill] (0.667,3.464) circle [radius =0.025];
  \draw[fill] (2,3.464) circle [radius =0.025];
  \draw[fill] (3.333,3.464) circle [radius =0.025];
  \draw[fill] (4.667,3.464) circle [radius =0.025];
  
 \end{tikzpicture}
\caption{Here are Newton polyhedra of $f_i$ in the case $n=2$. On each picture we highlight~$NP(f_i)$ with a blue colour.}
\end{center}
 \end{figure}  
 \begin{proof}
 Actually, consider the set $H'$ from the previous lemma. As for any $\Pi_j\in H'$, the restriction of function $F$ to $\Pi_j$ equals $M'=\max\{F(p_i)\}$, then
 \begin{equation}
  \Delta \subset \{ I|\ F(I)\le M'\}
 \end{equation}
 Therefore for any $\Pi_j \in H'$ the half-space $\Pi_j^+$ also intersects $\Delta$ by $\Delta(V')$.
 \end{proof}
 
In the language of equations this means that the first $s+1$ equations depend only on the first $s+1$ variables.

 \section{Proof of the theorem}
 From the previous section we deduce a useful corollary about morphisms between projective bundles:
 \begin{corollary}\label{subbundles}
 Assume $\phi:\p(E) \to \p(F)$ is a morphism over the base $B$ of degree $d>1$, such that its restriction to a fiber corresponds to an unstable orbit in $R^m(\p(V),\p(W))$. Then there are subbundles $E_0 \subsetneq E$ and 
 $F_0\subsetneq F$, such that 
 \begin{equation}
  \phi^{-1}(\p(F_0)) = \p(E_0)
 \end{equation}
 and $0<rk(E_0)=rk(F_0)<rk(E)=rk(F)$.
 \end{corollary}
 \begin{proof}
  By the results in the previous section, in any fiber of $\p(F)$ there are coordinates in which for any $0\le j\le s$
  \begin{equation}
   y_j = f_j(x_0, \dots, x_s)
  \end{equation}
  We claim that the preimage of the subspace $H=\{y_0=\dots=y_s=0\}$ is the subspace $\{x_0=\dots=x_s=0\}$. Indeed the last subspace is certainly contained in the preimage of the first one. If there is another point $P=(p_0:\dots p_s:p_{s+1}:\dots :p_n)$ in that 
preimage, consider 
the projective subspace generated by $P$ and the last $n-s$ base vectors: its dimension
is $n-s$, so it must have nonempty intersection with the subvariety given by the
equations 
\begin{equation}
f_{s+1}(x_0, \dots, x_n)=\dots = f_n (x_0, \dots, x_n)=0
\end{equation}
which has dimension at least $s$. Any point in this intersection must be an 
indeterminacy point of $f$, a contradiction.

These subspaces $H$ fit together in a subbundle $F_0 \subsetneq F$. The same happen
to their preimages, giving a subbundle $\phi^{-1}(F_0) = E_0 \subsetneq E$.   
 \end{proof}
 
 To complete the proof of the theorem let us consider a linear mapping induced by the morphism $\phi$:
 \begin{equation}
  \phi^*: F^* \to S^m E^*
 \end{equation}
 As we have shown we have subbundles $E_0$ and $F_0$, such that the following diagram commutes:
 \begin{equation}
  \xymatrix{
   0\ar[r] & (F\slash F_0)^* \ar[r] \ar[d]_{(\phi\slash \phi_0)^*} & F^* \ar[d]_{\phi^*} \ar[r] & F_0^* \ar[d]_{\phi^*_0} \ar[r] & 0 \\
   0\ar[r] &(S^m E\slash S^m E_0)^*  \ar[r] & S^m E^* \ar[r] &  S^m E_0^* \ar[r] & 0
   }
 \end{equation}
 Consider the bundle $(S^m E\slash S^m E_0)^*$ and write
 \begin{equation}
  (S^m E\slash S^m E_0)^*  \cong \oplus_{i=0}^{m-1} S^i E^*_0 \otimes S^{m-i} (E\slash E_0)^*
 \end{equation}
 In particular there is a projection
 \begin{equation}
   (S^m E\slash S^m E_0)^* \xrightarrow{pr_0} S^m(E\slash E_0)^*
 \end{equation}
and $pr_0 \circ (\phi\slash \phi_0)^* = \psi^*$ induces a map between projective bundles $\p(E\slash E_0)$ and $\p(F\slash F_0)$ given by degree $m$ polynomials. In fact 
this map is regular, that is, a morphism. To check this one observes that 
one may view $(x_0:\dots : x_s)$ and $(y_0:\dots : y_s)$ from corollary \ref{subbundles} as coordinates on the projectivization of the quotients, and the map of these projectivizations is then given by $f_0, \dots f_s$.
To say that this map has no indeterminacy point $(p_0:\dots: p_s)$ is the same as to say
that the preimage of $\p(F_0)$ from corollary \ref{subbundles} contains nothing but $\p(E_0)$.    
 
\begin{proof}[Proof of the Theorem \ref{Morphism}]
 We argue by induction on $n+1 = rk E$. If $rk E = 1$ then $E$ is already linear, so the base of induction is trivial. 
 
 Suppose now, that for all ranks less then $n+1$ the statement is true. The restriction 
of the morphism $\phi$ to a fiber gives us an element in $R^m(\p(V),\p(W))\slash (PGL(V) \times PGL(W))$. 
 
 If this element corresponds to a stable orbit in $R^m(\p(V),\p(W))$, then the argument in the proof of theorem 1 in \cite{A} proves that after a finite \'etale base change both $\p(E)$ and $\p(F)$ trivialize. 
 As the variety $B$ is simply-connected, there are no nontrivial \'etale base changes, so both $\p(E)$ and $\p(F)$ are trivial and hence split.
 
 If we get an unstable orbit, then by corollary \ref{subbundles} the bundles $E$ and $F$ sit in short exact sequences:
 \begin{equation}
  \begin{split}
   &0\to E_0\to E\to E\slash E_0\to 0 \\
   &0\to F_0\to F\to F\slash F_0\to 0
  \end{split}
 \end{equation}
 and there are morphisms given by polynomials of the same degree $m>1$ between the projectivisations of bundles $E_0$, $F_0$, $ E\slash E_0$ and  $F\slash F_0$, namely 
 \begin{equation} 
  \begin{split}
  &\phi_0: \p(E_0) \to  \p(F_0)\\
  &\psi: \p(E\slash E_0) \to \p(F\slash F_0)
  \end{split}
 \end{equation} 
 By the inductive assumption all these bundles must split into direct sums of line 
bundles. Since for any line bundle $\mathcal{L}$ on $B$, its first cohomology $H^1(B, \mathcal{L}) = 0$, we see that 
 \begin{equation}
  Ext^1(E\slash E_0, E_0)=Ext^1( F\slash F_0,F_0)=0
 \end{equation}
 So the extensions are trivial too. Consequently $E$ and $F$ split into a direct sum of line bundles.
\end{proof}

\medskip

{\bf Acknowledgements:} This paper has been prepared within the framework of a 
 subsidy granted to the NRU HSE Laboratory of Algebraic Geometry by the Government of the Russian Federation for the implementation of the Global 
 Competitiveness Program. The first-named author was partially supported by the 
Young Russian Mathematics award.

\noindent {\sc National Research University Higher School of Economics\\
Laboratory of Algebraic Geometry and Applications\\
6 Usacheva str., Moscow, Russia}

\end{document}